\documentclass[oneside,english]{amsart}
\usepackage[pdfencoding=auto,psdextra]{hyperref}
\usepackage[T1]{fontenc}
\usepackage[latin9]{inputenc}
\usepackage{geometry}
\usepackage{amstext}
\usepackage{amsthm}
\usepackage{amssymb}
\usepackage{graphicx}
\usepackage[outdir=./]{epstopdf}

\hypersetup{colorlinks,
	citecolor=red,
	linkcolor=blue,}
\usepackage{color}
\makeatletter
\numberwithin{equation}{section}
\numberwithin{figure}{section}
\theoremstyle{plain}
\newtheorem{thm}{\protect\theoremname}
  \theoremstyle{plain}
  
  \theoremstyle{definition}
  \newtheorem{defn}[thm]{\protect\definitionname}
  \theoremstyle{plain}
  
  \theoremstyle{plain}

\makeatother

\usepackage{babel}
  \providecommand{\corollaryname}{Corollary}

  \providecommand{\definitionname}{Definition}
  \providecommand{\lemmaname}{Lemma}
  \providecommand{\propositionname}{Proposition}
\providecommand{\theoremname}{Theorem}
  \providecommand{\remarkname}{Remark}

\newcommand{\Z}{\mathbb{Z}}
\newcommand{\R}{\mathbb{R}}
\newcommand{\q}{\mathbf{q}}
\newcommand{\w}{\mathbf{w}}
\newcommand{\p}{\mathbf{p}}

\newcommand{\e}{\mathbf{e}}

\newcommand{\0}{\mathbf{0}}
\newcommand{\E}{\mathbf{E}}

\newcommand{\W}{\mathbf{W}}
\newcommand{\pd}[1]{\frac{\partial}{\partial #1}}

\newcommand{\cbinom}[2]{\begin{Bmatrix} #1 \\ #2 \end{Bmatrix}}

\DeclareMathOperator{\vol}{vol}
\DeclareMathOperator{\Todd}{Todd}

\begin{document}

\title{A Continuous Analogue of Lattice Path Enumeration}

\author{Tanay Wakhare$^{\dag}$}
\address{$^{\dag}$~University of Maryland, College Park, MD 20742, USA}
\address{$^{\dag}$~National Institute for Biological and Mathematical Synthesis, Knoxville, TN 37996, USA}
\email{twakhare@gmail.com}
\author{Christophe Vignat$^{\ddag}$}
\address{$^\ddag$~Tulane University, New Orleans, LA 70118, USA}
\address{$^\ddag$~LSS/Supelec, Universite Paris Sud Orsay, France}
\email{cvignat@tulane.edu} 
\author{Quang-Nhat Le$^{\S}$}
\address{$^\S$~Brown University, Providence, RI 02912, USA}
\email{quang\_nhat\_le@brown.edu}
\author{Sinai Robins$^{\ast}$}
\address{$^\ast$~Brown University, Providence, RI 02912, USA}
\email{sinai.robins@gmail.com}

\begin{abstract}
Following the work of Cano and D\'iaz, we consider a continuous analog of lattice path enumeration. This allows us to define a continuous version of any discrete object that counts certain types of lattice paths. As an example of this process, we define continuous versions of binomial and multinomial coefficients, and describe some identities and partial differential equations they satisfy. Finally, we illustrate a general method to recover discrete combinatorial quantities from their continuous analogs.
\end{abstract}

\maketitle

\section{Introduction}

Cano and D\'iaz \cite{Diaz1, Diaz2} have recently explored a novel method of obtaining continuous analogues of discrete objects such as binomial coefficients and Catalan numbers. First, they realized these discrete quantities as the number of certain \textbf{lattice paths}. Next, they considered \textbf{directed paths} as continuous extensions of lattice paths and define \textbf{moduli spaces} of directed paths. Finally, they declared the volumes of these moduli spaces to be the continuous versions of the original discrete objects. 	

Here we extend some of Cano and D\'iaz's work to higher dimensions, obtaining a partial differential equation that the continuous multinomials satisfy, which generalizes the partial differential equation of Cano and D\'iaz from dimension $2$ to dimension $n$. Finally, we show how to use Todd operators to discretize these moduli spaces, enabling us to retrieve the number of lattice paths from their associated moduli space. 
	
One motivation for extending the study of lattice paths to a natural continuous analogue is that it is difficult to count discrete lattice paths satisfying various geometric constraints; it might be much easier to compute the volume of a related polytope (or union of polytopes) in some cases. This can in turn give us natural bounds for the number of lattice paths, and especially lattice paths under additional geometric constraints. 

We begin by describing the work of Cano and D\'iaz in some detail.  Consider a collection of vectors $\W=\{ \w_1,\dots,\w_N \}$  in $\Z^d$ which all lie on the same side of some fixed hyperplane containing the origin.  The vectors $\w_i$  are called \textbf{admissible directions}. 
	We define a \textbf{lattice path} as an ordered $(n+1)$-tuple of integer vectors
\begin{align}\label{latticepath}
&(\0, \p_1, \dots, \p_n), \ \text{ with each } \p_j \in \Z^d, \text{ and where } 
\end{align}
\begin{equation*}
 \p_k := \p_{k-1} + \lambda_k \w_{c_k},
\end{equation*}
for some $\w_{c_k} \in \mathbf W$, and $\lambda_k \in \Z_{\geq 0}$. Intuitively, a lattice path is a finite path in $\mathbb Z^d$ that follows (some of) the directions $\w_1,\dots,\w_N$ using integer steps.  The classical example of lattice path counting is the binomial coefficient $\binom{b}{a}$, which counts the number of lattice paths in $\R^2$ from the origin to a point $\q:= (a,b-a) \in \Z_{\geq 0}^2$, using the directions $\w_1:= (1,0)$ and $\w_2 := (0,1)$. 	It is also desirable to give a $1-1$ correspondence between each lattice path and the relevant $\lambda_k$ that define it, as is done in \eqref{polytope_of_directed_paths} below.
	
We would like to explore the space of all paths from the origin to some fixed $\q \in \R^d$, still using the admissible directions $\W$, but now using real coefficients.   Following Cano and D\'iaz, we
define a \textbf{directed path} using the same set-up as in definition
 \eqref{latticepath} above except for the important difference that now each coefficient $\lambda_i$ is a non-negative real number.  
Consider the set of all directed paths from the origin to a fixed $\q \in \R^d$, using the set of directions from the admissible directions $\W$.  That is, we define

\begin{equation}\label{polytope_of_directed_paths}
P(\q, \mathbf c) := \{  
(\lambda_1, \dots, \lambda_n) \in \R_{\geq 0}^n \ | \ 
 \lambda_1 \w_{c_1} + \dots + \lambda_n \w_{c_n} = \q
\},
\end{equation}
for some $\w_{c_1},   \dots,  \w_{c_n} \in \W$.  By definition,  $P(\q, \mathbf c)$ is
a polytope, which we call a \textbf{path polytope}. We call the collection of indices used here, namely $\mathbf c:= (c_1, \dots, c_n),$ a \textbf{pattern} for the directed paths.    It also follows easily from this definition that the path polytope 
$P(\q, \mathbf c) \subset \R^n$ has dimension $n-d$. Most importantly, we can interpret  the set of integer points in $P(\q, \mathbf c)$ as the set of {\em lattice paths}, with pattern $\mathbf c$,  defined by \eqref{latticepath}.  In other words, we define
\begin{equation}\label{lattice.points.in.the.path.polytope}
L(\q, \mathbf c) := \{P(\q, \mathbf c) \cap \Z^d \},
\end{equation}	
	the set of integer points in $P(\q, \mathbf c)$, which is also the set of 
lattice paths (from $\0$ to $\q$) that use the subset  $\w_{c_1},   \dots,  \w_{c_n}$       
of the admissible directions $\W$.

We next describe the moduli space of all directed paths from the origin to any 
$\q \in \R^d$.  In order to do so, we must first consider all ``words'' in the alphabet consisting of the ``directions'' given by $\W$.   Precisely,  let $D(n,N)$ be the set of  words of length $n$, in $N$ symbols, with no occurrence of two equal consecutive symbols appearing in any word.  Such words are also known as \textbf{Smirnov words}. In this context, we use the Smirnov words to keep track of the indices of the admissible directions in a given path, which each pattern $\mathbf c$ represents.
	
The \textbf{moduli space of all directed paths} from the origin to $\q \in \R^d$ is defined by the union of the following polytopes:
	\[ \mathcal{M}_{\W}(\q) = \coprod_{n=0}^{\infty} \coprod_{\mathbf c \in D(n,N)}
	P(\q, \mathbf c) \]
This moduli space can be endowed with a natural flat metric, which is one of the innovations in the work of Cano and D\'iaz.  This definition of the muduli space of directed paths also suggests a natural definition for the volume of the moduli space, namely:  
\[ 
\vol(\mathcal{M}_{\mathbf{W}}(\mathbf{q})) := \sum_{n=0}^{\infty} \sum_{\mathbf c \in D(n,N)} \vol P(\q, \mathbf c). 
\]

These definitions also appeared in \cite{Diaz2}.
To be concrete, we first demonstrate the novel extension of the classical binomial coefficients to continuous binomial coefficients. Namely, Cano and D\'iaz defined 
for each $\q:= (s,x-s)  \in \R_{\geq 0}^2$, with $0 < s < x$, the following continuous binomial coefficient:
\begin{equation}
\cbinom{x}{s} := 
\vol(\mathcal{M}_{\mathbf{W}}(\mathbf{q})) := \sum_{n=0}^{\infty} \sum_{\mathbf c \in D(n,2)} \vol P(\q, \mathbf c). 
\end{equation}
We note that here the dimension is $d=2$, the set of admissible directions is 
$\W := \{(1,0), (0,1)\}$, and so $N=2$.  
In addition, each path of length $n$ has a pattern $\mathbf c:= (c_1, \dots, c_n)$, and hence the dimension of the corresponding path polytope $P(\q, \mathbf c)$ has dimension $n-2$.

Moreover Cano and D\'iaz obtained the following interesting formula for the continuous binomial coefficients:
\begin{equation}
	\label{I0I1}
		\cbinom{x}{s} = 2I_{0}\left(2\sqrt{s\left(x-s\right)}\right)+\frac{x}{\sqrt{s\left(x-s\right)}}I_{1}\left(2\sqrt{s\left(x-s\right)}\right),
\end{equation}
where $I_{\nu}(z)$ denotes the modified Bessel function of the first kind.

We study the $d$-dimensional extension of this continuous binomial coefficient to continuous multinomial coefficients.  Suppose we consider all lattice paths from the origin $\0$ to any $\q \in \Z^d$, using the standard basis as the set of admissible directions  
$\E := \{ \e_1,\dots,\e_d \}$.  The classical fact here is that the number of such lattice paths equals the multinomial coefficient:

\[
 \begin{pmatrix} q_1+\dots+q_d \\ q_1\ ,\ \dots\ ,\ q_d \end{pmatrix} := \frac{(q_1+\dots+q_d)!}{q_1!\ \dots\ q_d!}.
\]

We fix any $\q \in \R_{\geq 0}^d$, and as before we consider all directed paths between the origin and $\q$.  Fixing a pattern $\mathbf c:= (c_1, \dots, c_n)$, we get a path polytope 
$P(\q, \mathbf c)$ with dimension $n-d$.  
 A natural definition for the {continuous multinomial} would then be:
\begin{equation}
		\begin{Bmatrix} x_1+\dots+x_d \\ x_1\ ,\ \dots\ ,\ x_d \end{Bmatrix} := \vol \left( \mathcal{M}_{\mathbf{E}}(\mathbf{x}) \right)= 
		 \sum_{n=0}^{\infty} \sum_{\mathbf c \in D(n,d)} \vol P(\q, \mathbf c). 
\end{equation}

However, here we must correct an irregularity in the work of Cano and  D\'iaz. There was a minor error in their calculation of $P(\q, \mathbf c)$ amounting to an extra multiplicative factor of $\sqrt d$,  where $d$ is the dimension of the ambient space.   This leads us to consider a new definition for the continuous multinomial coefficient, with more details appearing below in Section \ref{errors}.   Assuming this new definition of a continuous multinomial, we can now state our main results. We recall the 
Borel transform $\mathcal{B}(f)$, which acts on a univariate function
$f(x) = \sum_{i=0}^{\infty} k_i x^i$ (which must therefore be analytic at the origin) by the formula:
\[ 
\mathcal{B}(f) (x) := \sum_{i=0}^{\infty} \frac{k_i}{i!}\ x^i. 
\]

For a multi-variable analytic function $f(x) = \sum_{i_1,\dots,i_d=0}^{\infty} k_{i_1 \dots i_d} x^{i_1} \dots x^{i_d}$, we define similarly its Borel transform as
	\[ \mathcal{B}(f) (x_1,\dots,x_d) := \sum_{i_1,\dots,i_d=0}^{\infty} \frac{k_{i_1 \dots i_d}}{i_1! \dots i_d!}\ x^{i_1} \dots x^{i_d}. \]
	
\medskip
	
\begin{thm}\label{thm:ContMonomials}
Let
\begin{equation}\label{F}
F(x_1,\dots,x_d) := \frac{1}{1-\left( \frac{x_1}{1+x_1} + \dots + \frac{x_d}{1+x_d} \right)}, 
\end{equation}
which is analytic at $(0, \dots, 0)$. 
Then the continuous multinomial is equal to
\[ 
\begin{Bmatrix} x_1+\dots+x_d \\ x_1\ ,\ \dots\ ,\ x_d \end{Bmatrix} = \pd{x_1} \cdots \pd{x_d} \mathcal{B}(F) (x_1,\dots,x_d).
\]	
\end{thm}

\bigskip
	
Another interesting result of Cano and Diaz  \cite{Diaz2} is the following elegant and surprising identity for the continuous binomial coefficients (in dimension $2$):
\begin{equation}    \label{continuous.binomial}
 \pd{x}\pd{y}\cbinom{x+y}{x} = \cbinom{x+y}{x}. 
 \end{equation}

This appears the continuous analogue of the usual identity
\[
\Delta_{n} \Delta_{k} \binom{n+k}{k} = \binom{n+k}{k}
\]
for binomial coefficients, where $\Delta_{n}f\left(n\right)=f\left(n+1\right)-f\left(n\right)$ is the forward difference operator.
	
\medskip
\noindent
We obtain the following generalization of \eqref{continuous.binomial}, in the case of dimension $d$, for the multinomial coefficients. 	
	
	\begin{thm}\label{thm:PDIdentity}
		As a multi-variable function, the continuous multinomial satisfies the following partial differential equation:
		\begin{equation} 
			\prod_{j=1}^n \left(1+\pd{x_j}\right) \begin{Bmatrix} x_1+\dots+x_d \\ x_1\ ,\ \dots\ ,\ x_d \end{Bmatrix} = \sum_{i=1}^n \prod_{j \neq i} \left(1+\pd{x_j}\right) \begin{Bmatrix} x_1+\dots+x_d \\ x_1\ ,\ \dots\ ,\ x_d \end{Bmatrix}. \label{eq:PDIdentity}
		\end{equation}
	\end{thm}
	
	
\bigskip

\bigskip

	The  paper is organized as follows: in Section \ref{errors} we correct an error in \cite{Diaz1} and \cite{Diaz2}, and then motivate our definition of the continuous multinomial. In Section \ref{Section3} we prove Theorem \ref{thm:ContMonomials} and derive a closed form expression for the continuous multinomial in terms of the Borel transform. In Section \ref{Section:PDE} we prove Theorem  \ref{thm:PDIdentity} and extend the two dimensional PDE identity of Cano and D\'iaz. In Section \ref{Section5} we show how to recover discrete multinomials from continuous multinomials, and carry out the calculation in two dimensions.
	
\section{A correction and new definitions}\label{errors}

Here we point out that the papers of Cano and Diaz  [\cite{Diaz1}, \cite{Diaz2}] use a 
different  definition for volumes of simplices than the usual Riemannian definition of volume.  We now explain this discrepancy between the two definitions of the volume of a simplex.
	Consider the $n$-simplex 
\begin{equation}
\Delta_n^t := \{s_0, \ldots, s_n \in \R_{\geq 0} : s_0 + \cdots + s_n =t\},
\end{equation}
 which is embedded in $\R^{n+1}$. Now let $P_n^t$ be the convex hull of the origin and $\Delta_n^t$. We note that $\Delta_n^1$ is the convex hull of the standard basis $\{e_0, \ldots, e_n\}$ while $P_n^1$ is the convex hull of $\{0, e_0, \ldots, e_n\}$. The $(n+1)$ dimensional volume of $P_n^1$ is $\frac{1}{(n+1)!}$, but is also equal to
  $\frac{1}{n+1} \vol (\Delta_n^1) d(0, \Delta_n^1)$, where $\vol$ is the $n$-dimensional volume of $\Delta_n^1$ and $d(0, \Delta_n^1)$ is the distance from the origin to $\Delta_n^1$. It can be observed that $d(0,\Delta_n^1) = \frac{1}{\sqrt{n+1}}$. Therefore,
\begin{equation}
\vol(\Delta_n^1) = \frac{1}{(n+1)!} \frac{n+1}{d(0,\Delta_n^1)} = \frac{\sqrt{n+1}}{n!}
\end{equation}
and
\begin{equation}
\vol(\Delta_n^t) = \frac{t^n \sqrt{n+1}}{n!}.
\end{equation}
For instance, when $n=1$, $\Delta_n^1$ is the segment connecting $(1,0)$ and $(0,1)$ which has length $\frac{\sqrt{1+1}}{1!} = \sqrt{2}$.

The work \cite{Diaz1} instead calculates the volume differently. It uses the following parametrization of $\Delta_n^1$:
\begin{equation}
l_1 = s_0, l_2 = s_0 + s_1, \ldots, l_n = \sum_{i=0}^{n-1} s_i,
\end{equation}
and writes
\begin{equation}
\Delta_n^t = \{  (l_1, \ldots, l_{n}) \in \R^{n}: 0 \leq l_1 \leq \cdots \leq l_n \leq l_{n+1} = t \}.
\end{equation}
The paper then claims that $\vol(\Delta_n^t) = \frac{t^n}{n!}$. Although the coordinate change is linear (and is given by an upper triangular matrix) and does not alter the $(n+1)$-dimensional volume on $\R^{n+1}$, it does change the induced $n$-dimensional volume on most affine hyperplanes. The example above demonstrates this point for $n=1$.

Let $T$ denote the above coordinate change. We note that using $T$ will destroy the product structure which is crucial in the computation of the volume of the moduli space of directed paths. More specifically, given a product of two simplices $\Delta_1 \times \Delta_2$, its image $T(\Delta_1 \times \Delta_2)$ is not a product of two simplices. For example, take $\Delta_1$ and $\Delta_2$ to be two intervals ($1$-simplices). Then their product is a rectangle, which is transformed by $T$ into an parallelogram, which in general is not a product of $1$-simplices.

	We now discuss what this means for the continuous multinomial case. As a preliminary, the \textbf{frequency vector} $\nu(c)$ of a Smirnov word $c$ encodes the number of times each letter appears in $c$. Note that the coordinates of $\nu(c)$ sum up to $n$. We now consider
\begin{equation}
		\begin{Bmatrix} x_1+\dots+x_d \\ x_1\ ,\ \dots\ ,\ x_d \end{Bmatrix} := \vol \left( \mathcal{M}_{\mathbf{E}}(\mathbf{x}) \right):= 
		 \sum_{n=0}^{\infty} \sum_{\mathbf c \in D(n,d)} \vol P(\q, \mathbf c),
\end{equation}
the previously motivated definition of the continuous multinomial. We take $k = d$ and $\mathbf{W} = \mathbf{E} := \{ \mathbf{e}_1, \dots, \mathbf{e}_d \}$ the standard basis of $\R^d$. Let $\nu = (\nu_1, \dots, \nu_d) \in \Z_{\geq 0}^d$ be an integer vector with $\nu_1+\dots+\nu_d=n$ and $c$ a Smirnov word with frequency vector $\nu$. Then, we have the following identities:
		\begin{align}
			\mathcal{M}^c_{\mathbf{E}}(\mathbf{x}) &= \{ a_1,\dots,a_n \in \R_{\geq 0} \ : \ a_1 \mathbf{e}_{c_1} + \dots + a_n \mathbf{e}_{c_n} = \mathbf{x} \} \label{eq:Identity1}\\
				&= \{ a_1,\dots,a_n \in \R_{\geq 0} \ : \ \sum_{i_1:c_{i_1}=1} a_{i_1} = x_1, \dots, \sum_{i_d:c_{i_d}=d} a_{i_d} = x_d \}. \label{eq:Identity2}
		\end{align}
		In other words, the vector identity in \eqref{eq:Identity1} breaks into $d$ independent scalar identities in \eqref{eq:Identity2}. Therefore, $\mathcal{M}^c_{\mathbf{E}}(\mathbf{x})$ is the direct product of $d$ simplices isomorphic to $\Delta^{n_1}(x_1), \dots, \Delta^{n_d}(x_d)$, where
		\[ \Delta^m(y) := \{ b_1, \dots, b_m \in \R_{\geq 0} \ : \ b_1 + \dots + b_m = y \}. \]
		Stricly speaking, the volume of the simplex $\Delta^m(y)$ is $\frac{y^{m-1}\sqrt{m}}{(m-1)!}$.

\subsection{An alternate definition of volume}
	
If, instead of the usual Riemannian volume, we would use the Cano and Diaz modified volume measure, which we call 
 $\vol_{\text{CD}}$, and defined by:
 \begin{equation}\label{alternatevolume}
  \vol_{\text{CD}}(\Delta^m(y)) = \frac{y^{m-1}}{(m-1)!},
  \end{equation}
then our continuous binomial coefficient would coincide with that of Cano and D\'iaz.  
\noindent
Therefore, by the product rule of volumes of products of simplices, we would obtain 
the modified result \[ \vol_{\text{CD}}(\mathcal{M}^c_{\mathbf{E}}(\mathbf{x})) = \frac{x_1^{\nu_1-1}}{(\nu_1-1)!} \dots \frac{x_d^{\nu_d-1}}{(\nu_d-1)!}. \]
and one could interpret this as a \textit{definition} for the continuous multinomial coefficient, as follows:
\begin{defn}
Let $\nu$ denote the frequency vector of the Smirnov word $c$. Then 
\begin{equation}
		\begin{Bmatrix} x_1+\dots+x_d \\ x_1\ ,\ \dots\ ,\ x_d \end{Bmatrix} := \sum_{n=0}^{\infty} \sum_{\mathbf c \in D(n,d)} \frac{x_1^{\nu_1-1}}{(\nu_1-1)!} \dots \frac{x_d^{\nu_d-1}}{(\nu_d-1)!}.
\end{equation}
\end{defn}
We emphasize that this is the object considered in the Cano and  D\'iaz papers \cite{Diaz1} and \cite{Diaz2}, and in the two dimensional case this coincides with the continuous binomial of Cano and D\'iaz:
\begin{equation*}
	\label{I0I1}
		\cbinom{x}{s} = 2I_{0}\left(2\sqrt{s\left(x-s\right)}\right)+\frac{x}{\sqrt{s\left(x-s\right)}}I_{1}
		\left(2\sqrt{s\left(x-s\right)}\right).
\end{equation*}

While the continuous multinomial loses some of its geometric intuition and motivation, this renormalized volume
\eqref{alternatevolume} leads to an object with very interesting arithmetic properties. In addition, the non-normalized version with the square root correction terms does \textit{not} have a closed form in terms of hypergeometric functions, even in two dimensions.

There are two competing definitions for the continuous multinomial, which can lead to some confusion. The object $\begin{Bmatrix} x_1+\dots+x_d \\ x_1\ ,\ \dots\ ,\ x_d \end{Bmatrix}$ \textit{always} refers to the definition without a $\sqrt{n}$ term, so Section \ref{Section3} and \ref{Section:PDE} refer to the continuous multinomial \textit{without} the $\sqrt{n}$ term. However, Section \ref{Section5} deal directly to the simplices defining the continuous multinomial. This means that we are using the normal Lebesgue measure on $\R^n$, so we do not use the notation $\begin{Bmatrix} x_1+\dots+x_d \\ x_1\ ,\ \dots\ ,\ x_d \end{Bmatrix}$ anywhere. Therefore, if we were to calculate volumes we would include the $\sqrt{n}$ term.

\section{Continuous multinomials}\label{Section3}

	\begin{proof}[Proof of Theorem \ref{thm:ContMonomials}]		
Given a vector $\nu = (\nu_1, \dots, \nu_k) \in \Z_{\geq 0}^k$ such that $\nu_1 + \dots + \nu_N = n$, let $D(n,N;\nu)$ denote the subset of Smirnov words in $D(n,N)$ whose frequency vectors are all equal to $\nu$.	As shown by Flajolet and Sedgewick \cite{Flajolet}, the cardinality of $D(n,N;\nu)$ is the coefficient of $y_1^{\nu_1} \dots y_N^{\nu_k}$ in the power series representation of the rational function
	\[ F(y_1, \dots, y_N) = \frac{1}{1-\left( \frac{y_1}{1+y_1} + \dots + \frac{y_d}{1+y_N} \right)}. \]
	
We expand $F$ into a Taylor series about the origin:
	
\[ F(x_1,\dots,x_d) = \frac{1}{1 - \left( \frac{x_1}{1+x_1} + \dots + \frac{x_d}{1+x_d} \right)} = \sum_{\nu_1,\dots,\nu_d=0}^{\infty} f_{\nu_1,\dots,\nu_d} x_1^{\nu_1} \dots x_d^{\nu_d}, \]
		where $f_{\nu_1,\dots,\nu_d}$ counts the number of Smirnov words with frequency vector $(\nu_1,\dots,\nu_d)$, as mentioned above. Therefore,
		\begin{align*}
			\begin{Bmatrix} x_1+\dots+x_d \\ x_1\ ,\ \dots\ ,\ x_d \end{Bmatrix} &= \sum_{n=0}^{\infty} \sum_{\mathbf c \in D(n,d)} \frac{x_1^{\nu_1-1}}{(\nu_1-1)!} \dots \frac{x_d^{\nu_d-1}}{(\nu_d-1)!}  \\
				&= \sum_{\nu_1,\dots,\nu_d=0}^{\infty} f_{\nu_1,\dots,\nu_d} \frac{x_1^{\nu_1-1}}{(\nu_1-1)!} \dots \frac{x_d^{\nu_d-1}}{(\nu_d-1)!} \\
				&= \pd{x_1} \cdots \pd{x_d} \left( \sum_{\nu_1,\dots,\nu_d=0}^{\infty} f_{\nu_1,\dots,\nu_d} \frac{x_1^{\nu_1}}{\nu_1!} \dots \frac{x_d^{\nu_d}}{\nu_d!} \right)\\
				&= \pd{x_1} \cdots \pd{x_d}\mathcal{B}(F)(x_1,\dots,x_d),
		\end{align*}
\noindent		
which completes the proof.
	\end{proof}

	\bigskip
	
\noindent	
When $d=2$, we can set $(x_1,x_2) = (x,y)$ and compute
	\[ F(x,y) = (1+x+y) + \sum_{n=1}^{\infty} \left( x^n y^n + x^n y^{n+1} + x^{n+1} y^n \right), \]
	and 
	\[ \mathcal{B}(F)(x,y) = (1+x+y) + \sum_{n=1}^{\infty} \left( \frac{x^n}{n!} \frac{y^n}{n!} + \frac{x^n}{n!} \frac{y^{n+1}}{(n+1)!} + \frac{x^{n+1}}{(n+1)!} \frac{y^n}{n!} \right), \]
	Therefore, we retrieve the formula for the continuous binomials in \cite{Diaz2}:
	\begin{align*} \cbinom{x+y}{x} &= \begin{Bmatrix} x+y \\ x\ ,\ y \end{Bmatrix} = \pd{x}\pd{y} \mathcal{B}(F)(x,y) \\
	& = \sum_{n=0}^{\infty} \left( \frac{x^n}{n!} \frac{y^n}{n!} + \frac{x^n}{n!} \frac{y^{n+1}}{(n+1)!} + \frac{x^{n+1}}{(n+1)!} \frac{y^n}{n!} \right)\\
	& = 2 I_0\left(2 \sqrt{xy} \right) + \left(x+y\right) \frac{I_2\left(2 \sqrt{xy} \right)}{\sqrt{xy}}.
		\end{align*}

	\bigskip
	
\section{Partial differential identity}\label{Section:PDE}

\begin{proof}[Proof of Theorem \ref{thm:PDIdentity}]
	Our approach is inspired by the method of dynamic programming in computer science. 
	
	Let us write $\mathcal{M}(\mathbf{x}) = \mathcal{M}_{\mathbf{E}}(\mathbf{x})$. The moduli space $\mathcal{M}(\mathbf{x})$ can be decomposed into subspaces $\mathcal{M}^{x_1}_{n+1}(\mathbf{x}), \dots, \mathcal{M}^{x_d}_{n+1}(\mathbf{x})$ of directed paths of length $n+1$ whose last step are in directions $\mathbf{e}_1, \dots, \mathbf{e}_d$, respectively. 
	
	Consider the subspace $\mathcal{M}_{n+1}(\mathbf{x})$ of directed paths of length $n+1$ from $\mathbf{0}$ to $\mathbf{x}$ following the directions of the standard basis vectors $\mathbf{e}_1, \dots, \mathbf{e}_d$. As above, this subspace can be further decomposed into $d$ pieces $\mathcal{M}^{x_1}_{n+1}(\mathbf{x}), \dots, \mathcal{M}^{x_d}_{n+1}(\mathbf{x})$, where $\mathcal{M}^{x_i}_{n+1}(\mathbf{x})$ is the set of such directed paths whose last step follows the direction of $\mathbf{e}_i$.
	
	Let $1 \leq i \leq d$. Suppose the last step is of distance $x_i-s$ in direction $\mathbf{e}_i$. Then the corresponding slice of $\mathcal{M}^{x_i}_{n+1}(\mathbf{x})$ can be identified isometrically with the disjoint union of $\mathcal{M}^{x_j}_{n}(x_1, \dots, x_{i-1}, s, x_{i+1}, \dots, x_d)$ for $j \neq i$. Therefore, by Fubini's theorem, we have
	\[ \vol(\mathcal{M}^{x_i}_{n+1}(\mathbf{x})) = \int_0^{x_i} \sum_{j \neq i} \vol(\mathcal{M}^{x_j}_{n}(x_1, \dots, x_{i-1}, s, x_{i+1}, \dots, x_d)) ds, \]
	which implies,
	\begin{equation} 
		\pd{x_i} \vol(\mathcal{M}^{x_i}_{n+1}(\mathbf{x})) = \sum_{j \neq i} \vol(\mathcal{M}^{x_j}_{n}(\mathbf{x})). \label{eq:PDIdentity1}
	\end{equation}
	The base case is $n=d$ with $\vol(\mathcal{M}^{x_i}_{d}(\mathbf{x})) = (d-1)!$. By convention, for $n < d$, $\vol(\mathcal{M}^{x_i}_{n+1}(\mathbf{x})) = 0$. Summing equation \eqref{eq:PDIdentity1} over all $n \geq d$, we obtain the following identity:
	\begin{equation} 
	\pd{x_i} \vol(\mathcal{M}^{x_i}(\mathbf{x})) = \sum_{j \neq i} \vol(\mathcal{M}^{x_j}(\mathbf{x})). \label{eq:PDIdentity2}
	\end{equation}
	or equivalently, 
	\begin{equation} 
	\left(1+\pd{x_i}\right) \vol(\mathcal{M}^{x_i}(\mathbf{x})) = \sum_{j = 1}^d \vol(\mathcal{M}^{x_j}(\mathbf{x})) = \vol(\mathcal{M}(\mathbf{x})). \label{eq:PDIdentity3}
	\end{equation}
	This infers
	\begin{equation} 
	\prod_{j=1}^n \left(1+\pd{x_j}\right)  \vol(\mathcal{M}^{x_i}(\mathbf{x})) = \prod_{j \neq i} \left(1+\pd{x_j}\right) \vol(\mathcal{M}(\mathbf{x})). \label{eq:PDIdentity4}
	\end{equation}
	Summing this identity over  $1\leq i \leq n$, we obtain the desired identity.
\end{proof}

\bigskip \noindent
In the case that the dimension $d=2$ and $(x,y) = (x_1,x_2)$,  identity \eqref{eq:PDIdentity} becomes
	\[ 
	\left(1+\pd{x}\right) \left(1+\pd{y}\right) \cbinom{x+y}{x} = \left(1+\pd{x} + 1 + \pd{y}\right) \cbinom{x+y}{x}, 
	\]
which simplifies to the Cano and D\'iaz result 
\[ 
\pd{x}\pd{y} \cbinom{x+y}{x} = \cbinom{x+y}{x}. 
\]

\bigskip
\section{Recovering discrete objects}\label{Section5}
In this section, we retrieve discrete binomial coefficients from the continuous binomial case. This is due to a general result in lattice point counting, the \textit{Khovanskii-Pukhlikov theorem}. We describe the theorem, and then carry out a calculation involving it.

We begin with the fundamental \textbf{Todd operator}, which is defined to be the following differential operator:
\[ \Todd_h := \frac{d/dh}{1-e^{-d/dh}} = \sum_{k \geq 0} (-1)^k \frac{B_k}{k!} \left(\frac{d}{dh}\right)^k = 1 + \frac{1}{2}\frac{d}{dh} + \frac{1}{12} \left(\frac{d}{dh}\right)^2 - \frac{1}{720} \left(\frac{d}{dh}\right)^4 + \dots \]
Here, $B_{k}=B_{k}\left(0\right)$  are the Bernoulli numbers and $B_{k}\left(x\right)$ are the Bernoulli polynomials with generating function
\[
\sum_{k \geq 0} z^{k}\frac{B_{k}\left(x\right)}{k!} = \frac{ze^{zx}}{e^{z}-1}.
\]


We next consider \textbf{unimodular integral polytopes}, i.e. polytopes whose vertices have integer coordinates and whose vertex tangent cones are generated by some basis of $\Z^d$ (and hence simple). Given a polytope $P$ and vertex $\mathbf v$, we define the vertex tangent cone at $\mathbf v$ to be $$\{\mathbf v + \lambda (\mathbf y - \mathbf v): \mathbf y \in P, \lambda \in \R_{\geq 0} \}. $$
Suppose $P$ has the hyperplane description 
\[ P = \{ \mathbf{x} \in \R^d : \mathbf{A}\mathbf{x} \leq \mathbf{b} \},\]
where the column vectors of $\mathbf{A}$ are primitive integer vectors in $\Z^d$. We define the perturbed polytope
\[ P(\mathbf{h}) = \{ \mathbf{x} \in \R^d : \mathbf{A}\mathbf{x} \leq \mathbf{b} + \mathbf{h} \}, \]
for some small $\mathbf{h} = (h_1,h_2,\dots,h_m)$. We also define the multi-dimensional Todd operator
\[ \Todd_{\mathbf{h}} := \prod_{k=1}^m \Todd_{h_k}. \]
The fundamental role of Todd operators is highlighted by the \textbf{Khovanskii-Pukhlikov theorem} for a unimodular polytope $P$:
\[ \#(P \cap \Z^d) = \left. \Todd_{\mathbf{h}} \vol(P(\mathbf{h})) \right|_{\mathbf{h}=0}. \]
More generally,
\[ \sum_{\mathbf{x} \in P \cap \Z^d} \exp(\mathbf{x} \cdot \mathbf{z}) = \left. \Todd_{\mathbf{h}} \int_{P(\mathbf{h})} \exp(\mathbf{x} \cdot \mathbf{z}) d\mathbf{x} \right|_{\mathbf{h}=0}. \]
The assumption of unimodularity is important. Loosening it will require replacing the Todd operator with much more complicated differential operators, as shown in a version of Euler-Maclaurin formula for simple polytopes in \cite{KSW}.

Lattice paths are directed paths whose steps are only allowed to be integer multiples of one of the admissible directions. We now consider the binomial case. Given a combinatorial pattern, the moduli space of \textit{directed} paths with the given pattern, from $(0,0)$ to $(x,y)$, is a (direct) product of simplices of the form
\[ \Delta := \Delta^x_n := \{ a_1, \dots, a_n \geq 0, a_1 + \dots + a_n = x \}. \]
The space of \textit{lattice} paths, with the same pattern, from $(0,0)$ to $(x,y)$ is a (direct) product of \textit{discrete} simplices of the form
\[ \Lambda := \Lambda^x_n := \Delta^x_n \cap \Z_+^n = \{ a_1, \dots, a_n \in \Z, a_1, \dots, a_n > 0, a_1 + \dots + a_n = x \}. \]
Note that the simplex $\Delta$ is not full-dimensional, so the Khovanskii-Pukhlikov theorem does not apply directly and we do need a small trick to make it work. For a small $\mathbf{h} = (h_1,\dots,h_n,h_+,h_-)$, consider the perturbed simplex
\[ \Delta(\mathbf{h}) := \{ a_1 \geq h_1, \dots, a_n \geq h_n,\quad x - h_- \leq a_1 + \dots + a_n \leq x + h_+ \}, \]
whose volume is
\[ \vol(\Delta(\mathbf{h})) = \frac{1}{n!} \left( (x+h_+-h_1-\dots-h_n)^n - (x-h_--h_1-\dots-h_n)^n \right). \]
We expect
\[ \left. \Todd_{\mathbf{h}} \vol(\Delta(\mathbf{h})) \right|_{\mathbf{h}=0} = \# \Lambda = \binom{x-1}{n-1}. \]

Set $\mathbf{h}' = (h_1,\dots,h_n)$ and use the following notation.
\begin{align*}
\Delta' &:= \{ a_1, \dots, a_n \geq 0, a_1 + \dots + a_n \leq x \},\\
\Delta'(\mathbf{h}',h_+) &:= \{ a_1, \dots, a_n \geq 0, a_1 + \dots + a_n \leq x + h_+ \},\\
\Delta'(\mathbf{h}',h_-) &:= \{ a_1, \dots, a_n \geq 0, a_1 + \dots + a_n \leq x - h_- \},\\
\Lambda'^+ &:= \{ a_1, \dots, a_n \in \Z, a_1, \dots, a_n > 0, a_1 + \dots + a_n \leq x \},\\
\Lambda'^- &:= \{ a_1, \dots, a_n \in \Z, a_1, \dots, a_n > 0, a_1 + \dots + a_n < x \}.
\end{align*}
By a polarized version of the Khovanskii-Pukhlikov theorem in \cite{KSW}, we have
\begin{align*}
\left. \Todd_{(\mathbf{h}',h_+)} \vol(\Delta'(\mathbf{h}',h_+)) \right|_{\mathbf{h}'=0=h_+} &= \# \Lambda'^+ = \binom{x}{n}, \\
\left. \Todd_{(\mathbf{h}',h_-)} \vol(\Delta'(\mathbf{h}',h_-)) \right|_{\mathbf{h}'=0=h_-} &= \# \Lambda'^- = \binom{x-1}{n}.
\end{align*}
These identities can also be verified manually from the definition of the Todd opearator. 
Also, it is important to note that our simplex is unimodular. Otherwise, the Khovanskii-Pukhlikov theorem does not apply, and we will have to use a more complicated version by Karshon-Sternberg-Weitsman \cite{KSW}, which applies to all simple polytopes.

Therefore,
\begin{align*}
&\left. \Todd_{\mathbf{h}} \vol(\Delta(\mathbf{h})) \right|_{\mathbf{h}=0} \\ &= \left. \Todd_{(\mathbf{h}',h_+)} \vol(\Delta'(\mathbf{h}',h_+)) \right|_{\mathbf{h}'=0=h_+} - \left. \Todd_{(\mathbf{h}',h_-)} \vol(\Delta'(\mathbf{h}',h_-)) \right|_{\mathbf{h}'=0=h_-} \\ &= \# \Lambda'^+ - \# \Lambda'^- = \binom{x}{n} - \binom{x-1}{n} = \binom{x-1}{n-1} =  \# \Lambda.
\end{align*}


\section{Further remarks}

In general, applying the Khovanskii-Pukhlikov theorem to the simplices that define our continuous objects will recover the appropriate discrete objects.  The Khovanskii-Pukhlikov machinery also works in the case that the polytope is a simple polytope \cite{KSW}, not merely a unimodular polytope.  Thus one could apply our approach to simple polytopes in order to discretize them and perhaps obtain future discretization results of this flavor. 

In principle, one could begin with an arbitrary set of admissible directions which are not even necessarily integer vectors, and develop an analogous theory.

Another interesting direction for future research is the pursuit of an $L_1$-metric approach to volumes, since incorporating such an $L_1$ approach into the Cano-Diaz machine might yield very interesting results.


\end{document}